\documentclass[a4paper,10pt]{amsart}
\usepackage{amsmath,amsthm,amsfonts,amssymb,calrsfs}

\usepackage[all]{xy}
\usepackage{tikz-cd, tikz}

\usepackage{graphicx}
\usepackage{yfonts}
\usepackage{xcolor}

\usepackage{array}
\usepackage{multirow}
\usepackage{caption}

\usepackage{todonotes}

\usepackage[hyperindex=true, colorlinks=true, linkcolor=blue, filecolor=black,
citecolor=blue, urlcolor=blue,
pagebackref=true,
bookmarks=true, bookmarksopen=true, bookmarksnumbered=true,
bookmarksopenlevel=2, pdfborder={0 0 0}]{hyperref}
\usepackage{cleveref}

\theoremstyle{plain}{
    \newtheorem{theorem}{Theorem}[section]
    \newtheorem{lemma}[theorem]{Lemma}
    \newtheorem{corollary}[theorem]{Corollary}
    \newtheorem{proposition}[theorem]{Proposition}
    
}

\theoremstyle{definition}{
    \newtheorem{definition}[theorem]{Definition}
    \newtheorem{defn-thm}[theorem]{Definition-Theorem}

}

\theoremstyle{remark}{
    \newtheorem{remark}[theorem]{Remark}

}

\newcommand\Ima{\mathrm{Im}}

\newcommand\cC{\mathcal{C}}

\newcommand\cF{\mathcal{F}}
\newcommand\cE{\mathcal{E}}
\newcommand\cO{\mathcal{O}}

\newcommand\cG{\mathcal{G}}
\newcommand\cH{\mathcal{H}}

\newcommand\cK{\mathcal{K}}

\newcommand\cL{\mathcal{L}}

\newcommand\cI{\mathcal{I}}
\newcommand\cQ{\mathcal{Q}}

\DeclareFontFamily{U}{mathc}{}
\DeclareFontShape{U}{mathc}{m}{it}%
{<->s*[1.03] mathc10}{}
\DeclareMathAlphabet{\mathcal}{U}{mathc}{m}{it}

\DeclareMathOperator{\sHom}{\mathcal{H\mkern-3mu om}}

\newcommand\del{\partial}

\newcommand\K{\mathbb{K}}
\newcommand\R{\mathbb{R}}

\newcommand\Z{\mathbb{Z}}
\newcommand\N{\mathbb{N}}

\newcommand\E{\mathbb{E}}
\newcommand\LL{\mathbb{L}}

\newcommand\Hom{\mathrm{Hom}}
\newcommand\Spec{\mathrm{Spec}}

\newcommand\codim{\mathrm{codim}}

\newcommand\Supp{\mathrm{Supp}}

\newcommand\qCoh{\mathrm{qCoh}}
\newcommand\Coh{\mathrm{Coh}}                              

\begin{document}

\title[Decomposition of toric sheaves]{Decomposition of toric sheaves}

\address{Univ Brest, UMR CNRS 6205, Laboratoire de Mathématiques de Bretagne
Atlantique, France}

\author[C. Tipler]{Carl TIPLER}
\email{carl.tipler@univ-brest.fr}


\begin{abstract}
We show that the abelian category of coherent toric sheaves is generated, as a Serre subcategory, by pushforwards of rank $1$ reflexive sheaves supported on invariant subvarieties. As an application, we extend previous results by Klyachko on cohomology of toric vector bundles to coherent toric sheaves.
\end{abstract}

\maketitle

\section{Introduction}
\label{sec:intro}
Let $X$ be a normal toric variety with torus $T$.  Invariants of coherent sheaves over $X$  are often defined and computed by mean of resolutions by suitable sheaves of $\cO_X$-modules, typically finite locally free resolutions. In the smooth case, Klyachko constructed a resolution
\[
0\to \cK_n\to\ldots\to\cK_1\to\cK_0\to \cE\to 0
\]
for any locally free toric sheaf $\cE$ over $X$, where
each $\cK_i$ is equivariant and splits as a direct sum of line bundles \cite{Kly90} (a sheaf is called toric if it admits a $T$-equivariant structure, see Section \ref{sec:toric sheaves}). This resolution was then used to provide toric descriptions of cohomology groups and Chern classes for toric vector bundles. However, such locally free resolutions do not exist in full generality \cite{Payne09}.

In this paper, we undertake a different approach  by mean of a decomposition theorem. Denote by $\Coh^T(X)$ the abelian category of coherent toric sheaves. Recall that a Serre subcategory is a full and non-empty subcategory closed under subobjects, quotients, and extensions. We will say that an object $\cE\in\Coh^T(X)$ is \textit{elementary} if it is pure, with reduced and irreducible support, and restricts to a rank one reflexive sheaf on its support. Equivalently, there exists an invariant subvariety $\iota:V\to X$, a rank one reflexive sheaf $\cL$ on $V$ and an equivariant structure on $\iota_*\cL$ such that $\cE\simeq\iota_*\cL$.
\begin{theorem}
 \label{theointro:decomposition theorem}
 Let $X$ be a normal toric variety. The Serre subcategory of $\Coh^T(X)$ generated by elementary toric sheaves is $\Coh^T(X)$ itself.
\end{theorem}
The proof of this result is quite simple, and constructive, see Section \ref{sec:decomposition}. In practice, a useful application of Theorem \ref{theointro:decomposition theorem} is the following. We will say that a property $P$ for objects in $\Coh^T(X)$ is \textit{preserved in short exact sequences} if for any short exact sequence of coherent toric sheaves
\[
0\to \cE\to\cF\to\cG\to 0,
\]
if $P$ is satisfied by two of the sheaves $\cE$, $\cF$ or $\cG$, it is satisfied by the third one.
\begin{corollary}
\label{corintro:decomposition theorem}
Let $X$ be a normal toric variety.
 Assume that a property $P$ for objects in $\Coh^T(X)$ is satisfied by elementary toric sheaves and is preserved in short exact sequences. Then $P$ is satisfied by all coherent toric sheaves on $X$.
\end{corollary}
In the smooth case, by mean of the Koszul resolution for structure sheaves of invariant subvarieties, we can improve this corollary.
\begin{corollary}
 \label{corintro:decomposition theorem smooth case}
 Let $X$ be a smooth toric variety.
 Assume that a property $P$ for objects in $\Coh^T(X)$ is satisfied by toric invertible sheaves and is preserved in short exact sequences. Then $P$ is satisfied by all coherent toric sheaves on $X$.
\end{corollary}
Note that we don't assume (quasi)projectivity nor completeness of $X$ in Corollary \ref{corintro:decomposition theorem smooth case}, which prevents from using directly equivariant locally free resolutions or Klyachko's resolution to obtain the result.

The proof of Theorem \ref{theointro:decomposition theorem} actually carries over for coherent sheaves on normal varieties, leaving the toric assumptions, and replacing elementary toric sheaves by pushforwards of rank one reflexive sheaves supported on subvarieties. However its usefulness in the toric setting comes from two crucial facts. First, the orbit-cone correspondence implies that there is a finite number of invariant subvarieties in $X$. Second, these subvarieties are toric and normal, and their rank one reflexive sheaves are generated by a finite number of irreducible invariant Weil divisors. Those finiteness properties, together with the fan description of those objects, turn Corollaries \ref{corintro:decomposition theorem} and \ref{corintro:decomposition theorem smooth case} into powerful devices to obtain results about  toric sheaves.

We will illustrate this by extending Klyachko's results on cohomology of toric vector bundles over smooth toric varieties \cite{Kly90}  to coherent toric sheaves over normal toric varieties, cf Theorem \ref{theo:computing cohomology of toric sheaves}. While this generalization is likely to be familiar to experts, cf Remark \ref{rem:quasicoherent extension of cohomology for toric sheaves}, our approach appears to be new. Klyachko’s original proof relies on a dual complex together with Serre duality \cite[Theorem 4.1.1 and Remark 4.1.2]{Kly90}. This double application of duality may seem unnatural. In contrast, our approach is more direct;  Corollary \ref{corintro:decomposition theorem} allows us to reduce to the case of rank one reflexive sheaves, which in turn follows from elementary homology computations for cellular decompositions of the sphere. Furthermore, as we show in Section \ref{sec:alexander and serre dualities}, Serre duality for toric line bundles follows from Alexander duality for compact subsets of the sphere. This perspective leads to an interpretation of Serre duality for toric vector bundles on smooth complete toric varieties as a version of Alexander duality with local coefficient systems.

\subsection*{Organisation of the paper} We start recalling the necessary background on toric sheaves and $\Sigma$-families in Section \ref{sec:toric sheaves}. The proofs of Theorem \ref{theointro:decomposition theorem} and Corollaries \ref{corintro:decomposition theorem} and \ref{corintro:decomposition theorem smooth case} are then given in Section \ref{sec:decomposition}. Finally, the application to cohomology is derived in Section \ref{sec:cohomology}.

\subsection*{Acknowledgments} CT would like to thank Ronan Terpereau for stimulating discussions on the topic.   The author is partially supported by the grants MARGE ANR-21-CE40-0011 and BRIDGES ANR-FAPESP ANR-21-CE40-0017. He also beneficiated from the support of the French government “Investissements d'Avenir” program integrated to France 2030, under the reference ANR-11-LABX-0020-01.

\subsection*{AI declaration use} No LLM was used to produce the content of this article. We encourage the Association for Human Mathematics (AHM).

\subsection*{Conventions and Notations}
We work over an algebraically closed field $\K$ and follow notations from \cite{CLS,Per04} for toric varieties and toric sheaves. In particular, we denote by $T=N\otimes_\Z \K^*$ the torus of a toric variety $X$, with $N$ the rank $n$ lattice of its one-parameter subgroups. We set $M=\Hom_\Z(N,\Z)$ to be the character lattice. Assuming $X$ to be normal, $\Sigma$ will stand for its fan of strongly convex rational polyhedral cones in $N_\R=N\otimes_\Z \R$ (see \cite[Chapter 3]{CLS}). For each $\sigma\in\Sigma$, we let $\sigma^\vee$ denote the dual cone, and $\sigma^\perp$ its orthogonal - both for the duality pairing $\langle\cdot,\cdot\rangle$. The variety $X$ is then covered by the $T$-invariant affine varieties \[U_\sigma=\Spec(\K[S_\sigma]),\] where  $\sigma\in \Sigma$ and \[S_\sigma=M\cap\sigma^\vee.\] We denote by $\tau\leq\sigma$ the inclusion of a face $\tau$ in $\sigma$, and $\tau<\sigma$ if the inclusion is strict (recall then that we have $U_\tau\subset U_\sigma$).  For $m\in M$, we let $\chi^m\in\K[S_\sigma]$ be the associated regular function on $U_\sigma$. The orbit associated to a cone $\sigma\in\Sigma$ by the orbit-cone correspondence will be denoted by $O(\sigma)$, and its closure by $V(\sigma)=\overline{O(\sigma)}$ \cite[Section 3.2]{CLS}; when $\rho\in\Sigma(1)$ is a ray, we also set $D_\rho=V(\rho)$. Finally, $\Sigma(d)$ (resp. $\sigma(d)$) stands for the $d$-dimensional cones in $\Sigma$ (resp. the $d$-dimensional faces of $\sigma$).

\section{Background on toric sheaves}
\label{sec:toric sheaves}
We will briefly recall Perling's description of toric coherent sheaves in terms of $\Sigma$-families, and refer to \cite{Per04} for further details and proofs. Let $X$ be a normal toric variety with torus $T$ and fan $\Sigma$.

\subsection{Equivariant sheaves}
\label{sec:equiv sheaves definition}
Denote the regular torus action by $ \alpha : T\times X\to X$. Let $\pi_i:T\times X\to X$ be the projection onto the $i$-th factor and  $\mu:T\times T\to T$ the multiplication. We will also use $\pi_{23}:T\times T\times X\to T\times X$ the projection on the last two factors, and for any given variety $Y$, $1_Y:Y\to Y$  the identity morphism.
\begin{definition}
 \label{def:equivariant sheaf}
 Let $\cE$ be a sheaf of $\cO_X$-modules on $X$. A \textit{$T$-equivariant structure} on $\cE$ is an isomorphism
 \[
 \Phi : \alpha^*\cE\to\pi_2^*\cE
 \]
 satisfying the following cocycle condition :
 \begin{equation*}
  \label{eq:cocycle condition equiv sheaf}
 (\mu\times 1_X)^*\Phi=\pi_{23}^*\Phi\circ(1_T\times\alpha)^*\Phi.
 \end{equation*}
 A \textit{toric sheaf} on $X$ is a pair $(\cE,\Phi)$ where $\cE$ is a sheaf of $\cO_X$-modules and $\Phi$ a $T$-equivariant structure on $\cE$. 
\end{definition}
 We will often omit the isomorphism $\Phi$ in the notations, and simply denote by $\cE$ a toric sheaf. A morphism  between two toric sheaves $(\cE,\Phi)$ and $(\cF,\Psi)$ is a morphism $\theta : \cE \to \cF$ of sheaves of $\cO_X$-modules such that 
\begin{equation*}
 \label{eq:def T equiv morphism sheaves}
\pi_2^*\theta\circ\Phi=\Psi\circ\alpha^*\theta.
\end{equation*}
The class of toric sheaves together with their morphisms form a category. By restricting the class of objects to quasicoherent (resp. coherent) sheaves of $\cO_X$-modules, one obtains the abelian categories of $T$-equivariant quasicoherent sheaves (resp. of $T$-equivariant coherent sheaves), denoted $\qCoh^T(X)$ (resp. $\Coh^T(X)$).

\subsection{Sigma families}
\label{sec:Sigma families}
Given a toric quasicoherent sheaf $\cE$, the equivariant structure induces naturally a representation of $T$ on each space of sections $\Gamma(U_\sigma,\cE)$, which splits according to weight spaces by the complete reducibility theorem :
\[
\Gamma(U_\sigma,\cE)=\bigoplus_{m\in M} \, \Gamma(U_\sigma,\cE)_m,
\]
where $\Gamma(U_\sigma,\cE)_m$ stands for weight $m$ sections. This induces an $M$-graded $\K[S_\sigma]$-module structure on $\Gamma(U_\sigma,\cE)$.  For $m\in S_\sigma$, multiplication by the regular function $\chi^m$ defines a degree-$m$ endomorphism of $\Gamma(U_\sigma,\cE)$. Hence, for each pair $(m,m')\in M$, the following
\[
\begin{array}{cccc}
 \chi^{m-m'}: & \Gamma(U_\sigma,\cE)_{m'} & \to &\Gamma(U_\sigma,\cE)_{m}\\
        &  f_{m'} & \mapsto & \chi^{m-m'}\cdot f_{m'}
\end{array}
\]
is well defined provided that $m-m'\in S_\sigma=\sigma^\vee\cap M$, that we will denote by $m'\leq_\sigma m$ ($\leq_\sigma$ descends to a partial order on $M/(\sigma^\perp\cap M)$). Setting
\[
E^\sigma_m:=\Gamma(U_\sigma,\cE)_m,
\]
and,  for $m\leq_\sigma m'$,
\[
\chi_{m,m'}^\sigma:=\chi^{m'-m},
\]
the collection of vector spaces $(E^\sigma_m)$ and maps $(\chi^\sigma_{m,m'})_{m\leq_\sigma m'}$ defines a \textit{$\sigma$-family} $E^\sigma$ in Perling's terminology \cite{Per04}  (we drop the morphisms $\chi^\sigma$ to ease notations).  Consider now $\sigma\in \Sigma$ and $\tau\leq\sigma$. The injection
\[
\iota_{\tau\sigma}:U_\tau\to U_\sigma
\]
has associated morphism of $\K$-algebras given by localisation
\[
\iota_{\tau\sigma}^*:\K[S_\sigma]\to\K[S_\tau]\simeq\K[S_\sigma]_{\chi^m},
\]
where $m\in\sigma^\vee$ is chosen so that $S_\tau=S_\sigma+\N\cdot(-m)$. Let $(\cE,\Phi)$ be a toric sheaf on $U_\sigma$. As the torus action commutes with the inclusion $\iota_{\tau\sigma}$, $\iota_{\tau\sigma}^*\cE$ naturally inherits an equivariant structure $\iota_{\tau\sigma}^*\Phi:=(1_T\times\iota_{\tau\sigma})^*\Phi$. When $\cE$ is quasicoherent, if $E^\sigma$ stands for the $\sigma$-family associated to $(\cE,\Phi)$, we will denote $\iota_{\tau\sigma}^*E^\sigma$ the  $\tau$-family associated to $\iota_{\tau\sigma}^*(\cE,\Phi)$. In a similar way, any morphism $\cE\to\cF$ between two quasicoherent toric sheaves on $U_\sigma$ yields a morphism $\phi^\sigma:E^\sigma\to F^\sigma$ between the associated $\sigma$-families that can be pulled back to a morphism $\iota_{\tau\sigma}^*\phi : \iota_{\tau\sigma}^*E^\sigma\to\iota_{\tau\sigma}^*F^\sigma$. These observations yield to $\Sigma$-families.
\begin{definition}
 \label{def:Sigma families}
 A \textit{$\Sigma$-family} for a fan $\Sigma$ is a pair $(E^\Sigma,\eta^\Sigma)$ consisting of a collection $E^\Sigma=(E^\sigma)_{\sigma\in\Sigma}$ of $\sigma$-families and a collection $\eta^\Sigma=(\eta_{\tau\sigma})_{\tau\leq \sigma}$ of isomorphisms $\eta_{\tau\sigma} : \iota_{\tau\sigma}^*E^\sigma \to E^\tau$ such that
 \begin{enumerate}
  \item for each $\sigma\in\Sigma$, $\eta_{\sigma\sigma}=1_{E^\sigma}$,
  \item for each $\rho\leq \tau\leq \sigma$, $\eta_{\rho\sigma}=\eta_{\rho\tau}\circ \iota_{\rho\tau}^*\eta_{\tau\sigma}$.
 \end{enumerate}
 A \textit{morphism of $\Sigma$-families} $\Phi^\Sigma : (E^\Sigma,\eta^\Sigma)\to (F^\Sigma,\nu^\Sigma)$ is a collection of morphisms $(\phi^\sigma : E^\sigma \to F^\sigma)_{\sigma\in\Sigma}$ of $\sigma$-families such that for all $\tau\leq\sigma$,
 \[
 \nu_{\tau\sigma}\circ \iota_{\tau\sigma}^*\phi^\sigma=\phi^\tau\circ \eta_{\tau\sigma}.
 \]
\end{definition}
The following is due to Perling \cite{Per04}.
\begin{theorem}
 \label{theorem:equiv categories Tsheaves and Sigma families}
 For a given fan $\Sigma$, the $\Sigma$-families together with their morphisms form a category that is equivalent to $\qCoh^T(X_\Sigma)$.
\end{theorem}
Coherence of $\cE\in\Coh^T(X)$ translates into \textit{finiteness} of the associated $\Sigma$-family by \cite[Proposition 5.11]{Per04}.
\begin{definition}
 \label{def:finite sigma families}
 A $\sigma$-family $(E^\sigma,\chi^\sigma)$ is \textit{finite} if it satisfies the following :
 \begin{enumerate}
  \item For each $m\in M$, $E^\sigma_m$ is finite dimensional,
  \item For each strictly increasing sequence $(m_i)_{i\in\Z}\in M^\Z$, i.e. such that 
  \[
\ldots<_\sigma m_{i-1} <_\sigma m_i <_\sigma m_{i+1} <_\sigma \ldots,
\] there exists $i_0\in\Z$ such that $E^\sigma_{m_i}=\lbrace 0\rbrace$ for $i\leq i_0$,
  \item For each $m\in M$, denote by $C^\sigma_m$ the cokernel of the map 
  \[
 \sum_{m'<_\sigma m} \chi^\sigma_{m',m} : \bigoplus_{m'<_\sigma m} E^\sigma_{m'} \to E^\sigma_m.
  \]
  Then the set of classes $[m]\in M/(\sigma^\perp\cap M)$ such $C^\sigma_m\neq\lbrace 0\rbrace$ is finite.
 \end{enumerate}
A $\Sigma$-family $(E^\Sigma,\eta^\Sigma)$ is \textit{finite} if for each $\sigma\in\Sigma$, $E^\sigma$ is finite.
\end{definition}
 Finally, when $\cE$ is torsion-free, the morphisms $(\chi^\sigma_{m,m'})$ and $(\eta_{\tau\sigma})$ provide injections of any space $E^\sigma_m$ of the associated $\Sigma$-family to a finite dimensional limit $E=\displaystyle\lim_{\to} E^\sigma_m$. In that case, we may view the $\Sigma$-family as collections of vector subspaces $E^\sigma_m\subset E$ and replace the maps $\chi^\sigma_{m,m'}$ by inclusions. This corresponds to \textit{families of multifiltrations}, which is the data of vector subspaces $(E^\sigma_m)_{\sigma\in\Sigma,m\in M}$ for a given finite dimensional space $E$, satisfying finiteness as in Definition \ref{def:finite sigma families}, together with the compatibility conditions :
 \begin{enumerate}
  \item For any pair $m\leq_\sigma m'$, $E^\sigma_m\subset E^\sigma_{m'},$
   \item  $E$ is the direct limit of the direct system $(E^\sigma_m)$ over $(M,\leq_\sigma)$,
\item For any face $\tau\leq\sigma$, and any $m\in M$, $E^\sigma_m\subset E^\tau_m$. Moreover, if $\tau=m_0^\perp\cap\sigma$ for $m_0\in\sigma^\vee$, there is $i_m\in\Z$ such that  for all $i\geq i_m$, $E^\tau_m= E^\sigma_{m+im_0}.$
 \end{enumerate}
 By \cite[Sections 5.3 and 5.4]{Per04}, families of multifiltrations correspond to torsion-free coherent toric sheaves.

\subsection{Examples : rank one reflexive sheaves}
\label{sec:examples}
Let $\cL$ be a rank one reflexive sheaf on $X$.
From \cite[Chapter 8.0]{CLS}, there are integers $(a_\rho)_{\rho\in\Sigma(1)}$ such that
 \[
 \cL\simeq \cO_X\big(\sum_{\rho\in\Sigma(1)} a_\rho\,D_\rho\big).
 \]
 In particular, $\cL$ is isomorphic to a $T$-equivariant rank $1$ reflexive sheaf.
\begin{proposition}
 \label{prop:family of filtrations rank one reflexive sheaves}
 Let $(L^\sigma_m)$ be the family of multifiltrations for $\cL$. Then for any $\sigma\in\Sigma$,
 \[
 L^\sigma_m = \left\{
 \begin{array}{ccc}
\lbrace 0\rbrace & \mathrm{ if } & \exists\; \rho\in\sigma(1)\:\vert\:   \langle m, u_\rho \rangle < -a_\rho \\
 \K             & \mathrm{ if } & \forall\; \rho\in\sigma(1)\:\vert\:   \langle m, u_\rho \rangle\geq -a_\rho.
 \end{array}
 \right.
 \]
\end{proposition}
Of course, there is a slight abuse of terminology in the above proposition, as 'the' family of multifiltrations is only defined up to  isomorphisms.
\begin{proof}
Let $\rho\in\Sigma(1)$. By definition,
\[
\Gamma(U_\rho,\cL)=\lbrace f\in\K(X)^*\:\vert\: \mathrm{div}(f)_{\vert U_\rho}+ a_\rho D_\rho \geq 0\rbrace \cup\lbrace 0\rbrace.
\]
Let $m\in M$. Then, $\mathrm{div}(\chi^m)$ is $T$-invariant. By the orbit-cone correspondence, $\mathrm{div}(\chi^m)_{\vert U_\rho}$ is supported on $D_\rho$ and the valuation of $\chi^m$ is given by $\nu_\rho(\chi^m)=\langle m , u_\rho \rangle$ \cite[Proposition 4.1.1]{CLS}. Hence, $\mathrm{div}(\chi^m)_{\vert U_\rho}+ a_\rho D_\rho \geq 0$ if and only if $\langle m , u_\rho \rangle+a_\rho\geq 0$. We deduce that the degree-$m$ component of $\Gamma(U_\rho,\cL)$ is given by
\[
\Gamma(U_\rho,\cL)_m=\left\{
 \begin{array}{ccc}
\lbrace 0\rbrace & \mathrm{ if } & \langle m , u_\rho \rangle < -a_\rho \\
 \K  \cdot\chi^m           & \mathrm{ if } & \langle m , u_\rho \rangle\geq -a_\rho.
 \end{array}
 \right.
\]
The result follows, as by definition:
\begin{equation*}
 \begin{array}{ccc}
L^\sigma_m & \simeq &  \Gamma(U_\sigma,\cL)_m\\
 & \simeq & \displaystyle\bigcap_{\rho\in\sigma(1)} \, \Gamma(U_\rho,\cL)_m
 \end{array}
\end{equation*}
where for the last isomorphism we used reflexivity of $\cL$ which implies normality, cf \cite[Proposition 1.6]{Har80}.
\end{proof}

\section{Decomposition theorem}
\label{sec:decomposition}

\subsection{Elementary toric sheaves}
Recall that for a coherent sheaf $\cE$ on a normal variety $X$, we define the dimension of $\cE$ to be
\[\dim(\cE):=\dim(\Supp(\cE)).\] 
The sheaf $\cE$ is then said to be \textit{pure} (of dimension $d=\dim\cE$) if all its non-trivial coherent subsheaves have dimension $d$. 
\begin{definition}
 \label{def:elementary toric sheaf}
 A coherent toric sheaf $\cE$ on a normal toric variety $X$ is \textit{elementary} if it is pure, with reduced and irreducible support, and restricts to a rank one toric reflexive sheaf on its support.
\end{definition}
Let $\cE$ be an elementary toric sheaf on $X$, with reduced support $\iota : V \to X$. Then, one has
$\cE\simeq\iota_* \iota^*\cE$ and by definition, there exists a rank one reflexive toric sheaf $\cL$ on $V$ such that $\cE\simeq\iota_*\cL$. The latter is an isomorphism of sheaves. Note however that the pushforward $\iota_*$ does not provide a natural equivariant structure on $\cE$. Hence, an elementary toric sheaf is specified by an invariant subvariety $\iota : V\to X$, a rank $1$ reflexive sheaf $\cL$ on $V$, \textit{and} an equivariant structure on $\iota_*\cL$. More precisely, in the simplicial case, we have :
\begin{lemma}
 \label{lem:elementary toric sheaf geometric}
 Assume that $X$ is simplicial. Let $\cE$ be an elementary toric sheaf on $X$. Then there are a cone $\tau\in\Sigma(d)$,
  integers $(a_\rho)_{\rho+\tau\in\Sigma(d+1)}$, and a weight $m\in M$ such that 
  \begin{equation}
   \label{eq:elementary toric sheaf geometric}
  \cE\simeq \Big(\big(\iota_*\cO_{V(\tau)}\big)\otimes_{\cO_X}\cO_X\big(\sum_{\tau+\rho\in\Sigma(d+1)}a_\rho \, D_\rho\big) \Big)[m],
  \end{equation}
where $\iota : V(\tau)\to X$ stands for the inclusion map, and $[m]$ denotes the shifted equivariant structure.
\end{lemma}
In the previous statement, the equivariant structure on $\iota_*\cO_{V(\tau)}$ is induced by the exact sequence
\[
0\to \cI_{V(\tau)} \to \cO_X\to \iota_*\cO_{V(\tau)}\to 0
\]
where $\cI_{V(\tau)}\subset\cO_X$ naturally inherits a toric structure by $T$-invariance of $V(\tau)$. The sum in
\[\cO_X\big(\sum_{\tau+\rho\in\Sigma(d+1)}a_\rho \, D_\rho\big) \]
is over rays $\rho\in\Sigma(1)$ such that $\tau+\rho\in\Sigma(d+1)$. Finally, $[m]$ stands for the shift $\Phi\mapsto \chi^m\cdot\Phi$ of the equivariant structure  by $\chi^m$.  
\begin{proof}
 By definition, there exists $V=V(\tau)$, for $\tau\in\Sigma(d)$, and a rank one reflexive sheaf $\cL$ on $V$ such that 
 $\cE\simeq\iota_*\cL$. We can find integers $(a_{\overline\rho})_{\overline\rho\in\Sigma_V(1)}$ such that
 \[
 \cL\simeq\cO_V(\sum_{\overline\rho\in\Sigma_V(1)} a_{\overline\rho} D_{\overline\rho}),
 \]
 where the $D_{\overline\rho}$'s stand for invariant divisors in $V$.
 As $X$ is simplicial, by the orbit-cone correspondence, one has a bijection between $T_X$-invariant divisors $D_\rho$ such that $\tau+\rho\in\Sigma(d+1)$ and $T_V$-invariant divisors on $V$ \cite[Theorem 3.2.6, Proposition 3.2.7]{CLS}. More precisely, for any $\rho\in\Sigma(1)$,
 \[
 D_\rho\cap V(\tau)=\big(\bigcup_{\rho\leq\sigma} O(\sigma)\big) \cap\big(\bigcup_{\tau\leq \sigma'} O(\sigma') \big)=\bigcup_{\rho+\tau\leq\sigma} O(\sigma),
 \]
 and the right hand side defines a divisor in $V(\tau)$ precisely when $\rho+\tau\in \Sigma(d+1)$. For $\rho\in\Sigma(1)$ such that $\tau+\rho\in\Sigma(d+1)$, we set $a_\rho:=a_{\overline{\rho}}$, where $\overline{\rho}$ is the projection of $\tau+\rho\in\Sigma$ to $\Sigma_V$ through the map
 \[
 N_\R\to N_\R/\mathrm{Span}_\R(\tau).
 \]
 As it is supported on $V$, we deduce that the sheaf 
 \[
 \big(\iota_*\cO_{V}\big)\otimes_{\cO_X}\cO_X\big(\sum_{\tau+\rho\in\Sigma(d+1)}a_\rho \, D_\rho\big)
 \]
 is isomorphic, as a coherent sheaf, to 
 \[
 \iota_*\Big(\cO_V(\sum_{\overline\rho\in\Sigma_V(1)} a_{\overline\rho} D_{\overline\rho})\Big)\simeq\iota_*\cL\simeq\cE.
 \]
 We claim that $\cE$ is simple, which implies the result, as all $T$-equivariant structures on a simple toric sheaf $(\cE,\Phi)$ are of the form $\chi^m\cdot \Phi$, for $m\in M$. To prove the claim, it is enough to notice that as $\iota_*\cO_V$ is supported on $V$, so is $\sHom_{\cO_X}(\iota_*\cO_V,\iota_*\cO_V)$, and then
 \[
 \sHom_{\cO_X}(\cE,\cE)\simeq \sHom_{\cO_X}(\iota_*\cO_V,\iota_*\cO_V)\simeq \iota_*\sHom_{\cO_V}(\cO_V,\cO_V).
 \]
We deduce that
 \[
 \mathrm{Hom}(\cE,\cE)\simeq \mathrm{Hom}(\cO_V,\cO_V)\simeq \K
 \]
 and the result follows.
\end{proof}
\subsection{Proof of Theorem \ref{theointro:decomposition theorem} and its corollaries}
\begin{definition}
 \label{def:serre subcategory}
 Let $\mathfrak{A}$ be an abelian category. A full subcategory $\mathfrak{B}$ is a \textit{Serre subcategory} if it contains $0$ and for any exact sequence
 \[
 E_1\to E_2\to E_3
 \]
 of objects in $\mathfrak{A}$, if $E_1$ and $E_3$ are in $\mathfrak{B}$, so is $E_2$.
\end{definition}
Taking $E_1=0$ in the definition, one sees that a Serre subcategory is closed under subobjects. Taking $E_3=0$ gives closeness under quotients. Hence, a Serre subcategory is a full and non-empty subcategory closed under subobjects, quotients, and extensions. While the statement of Theorem \ref{theointro:decomposition theorem} uses the categorical language, we will rather refer to it as a \textit{decomposition theorem} for toric sheaves. Indeed, we will see that the proof is constructive, and explains how to recover any coherent toric sheaf by successive short exact sequences, starting from elementary toric sheaves. In the proof, we will use the \textit{torsion filtration} of a coherent sheaf $\cE$ \cite[Chapter I, Section 1]{HuLe}. This is the unique filtration
\[
0\subset T_0(\cE)\subset T_1(\cE)\subset\ldots\subset T_d(\cE)=\cE
\]
where for each $0\leq i\leq d$, $T_i(\cE)\subset \cE$ is the maximal coherent subsheaf of dimension less or equal to $i$. When $\cE$ is toric, so are the $T_i(\cE)$'s by uniqueness, and by definition, $T_i(\cE)/T_{i-1}(\cE)$ is either zero, or pure and $i$-dimensional.

\begin{proof}[Proof of Theorem \ref{theointro:decomposition theorem}]
 It is enough to show that the smallest class of objects $\mathfrak{A}$ of $\Coh^T(X)$ containing elementary toric sheaves and stable under subobjects, quotients, and extensions, is the set of all coherent toric sheaves. Taking the full subcategory associated to this class gives the result. Let $\cE$ be a coherent toric sheaf. We will show that $\cE\in \mathfrak{A}$ by using a step by step decomposition of $\cE$ involving a finite number of short exact sequences. First, one can consider the torsion subsheaf $T(\cE)=T_{n-1}(\cE)\subset \cE$. We have a short exact sequence of toric sheaves :
 \[
 0\to T(\cE)\to \cE \to \cQ \to 0
 \]
 where the quotient $\cQ$ is torsion-free and equivariant. Hence, it is enough to prove that $\mathfrak{A}$ contains all torsion toric sheaves and all torsion-free toric sheaves. 
 
 We start by the torsion case, and use an induction on the dimension of the support. If $\cE$ is torsion and of dimension zero, then it is supported on the torus fixed points of $X$. By induction, we may assume that $\cE$ is supported on a single torus fixed point $p\in X$. Denoting $\iota: p\to X$ the inclusion, we have a short exact sequence of coherent toric sheaves
 \[
0\to \cI_p\cdot\cE\to\cE\to\iota_*\iota^*\cE\to 0
 \]
 supported on $p$. The multiplicity of $p$ in $\cI_p\cdot\cE$ is strictly smaller than in $\cE$, and by induction on $\mathrm{mult}_p(\cE)$ we may assume that $\cE\simeq \iota_*\iota^*\cE$. Then, by coherence, $\cE$ is equivalent to a finite dimensional $T$-representation over $p$. It is then a direct sum of elementary toric sheaves by the complete reducibility theorem, and the result follows in that case. Assuming the result for toric torsion sheaves of supports of dimension less or equal to $d-1$, we consider $\cE\in\Coh^T(X)$ and assume that it is torsion, with $\dim(\Supp(\cE))=d$. Using the torsion filtration for $\cE$, we can assume further that $\cE$ is pure. Then, we proceed by induction on the number of irreducible components for $\Supp(\cE)$. Let $V\subset \Supp(\cE)$ be such a component, and denote by $\iota : V\to X$ the inclusion. Then we have a short exact sequence :
 \[
 0\to \cH^0_V(\cE)\to \cE \to\cQ \to 0
 \]
 where $\cH^0_V(\cE)$ is the maximal coherent subsheaf supported on $V$. By maximality, $\cH^0_V(\cE)$ is unique, which implies by invariance of $V$ that it is toric. Then, the quotient $\cQ$ is either zero, or a toric pure sheaf of dimension $d$ whose support has less irreducible components than the one of $\cE$. By induction, we can restrict to the case when $\cE$ is pure with irreducible support. We then have a sequence of coherent toric sheaves
 \[
0\to \cI_V\cdot\cE\to\cE\to\iota_*\iota^*\cE\to 0
 \]
 supported on $V=\Supp(\cE)$. As the multiplicity of $V$ in $\cI_V\cdot \cE$ is strictly smaller than in $\cE$, by induction on the multiplicity, we may assume $\cE\simeq\iota_*\iota^*\cE$, that is the support of $\cE$ is reduced.  We then consider $\cE_{\vert V}$. This admits a short exact sequence (as toric sheaves on $V$) :
 \[
 0\to \cE_{\vert V} \to \cE_{\vert V}^{\vee\vee} \to \cQ' \to 0
 \]
 where this time the quotient $\cQ'$ is a torsion toric sheaf  on $V$. By our induction hypothesis on the dimension of the support, $\iota_*\cQ'\in \mathfrak{A}$, and thus we can restrict to the case when $\cE_{\vert V}$ is reflexive. Finally, we argue by induction on the rank. If the rank $r$ of $\cE_{\vert V}$ is one, then $\cE$ is an elementary toric sheaf, and thus $\cE\in\mathfrak{A}$. We can always find a saturated rank $1$ toric subsheaf $\cL$ of $\cE_{\vert V}$. Indeed, if $(E^\sigma_m)$ is the family of multifiltrations for $\cE_{\vert V}$, with $E=\displaystyle\lim_{\to} E^\sigma_m$, then the family of multifiltrations $(L\cap E^\sigma_m)$, for $L\subset E$ a line, corresponds to such a saturated subsheaf $\cL\subset \cE_{\vert V}$ (see e.g. \cite[Lemma 2.15]{NaTip}). The sheaf $\cL$ is reflexive on $V$, hence elementary on $X$, and the quotient $\cQ_\cL$ of $\cE_{\vert V}$ by $\cL$ is torsion-free, of rank $r-1$. Reducing again to its reflexive hull $\cQ_\cL^{\vee\vee}$ (the quotient is again a torsion sheaf on $V$), this concludes the proof in that case by induction on the rank.
 
 We move on to the torsion-free case. Assume  that $\cE$ is torsion-free. Then it lies in the canonical sequence
 \[
 0\to \cE \to \cE^{\vee\vee} \to \cQ \to 0,
 \]
 where $\cQ$ is a toric torsion sheaf. By the previous case, $\cQ\in \mathfrak{A}$. Hence, we reduced to the reflexive case, and a similar argument as before on the rank concludes the proof of the theorem.
\end{proof}
\begin{remark}
 By adapting the proof of Theorem \ref{theointro:decomposition theorem}, one can show that any equivariant injection $\cE\to \cF$ between two torsion-free toric sheaves of the same rank factorizes as
 \[
\cE\to \cE_1\to\ldots\to \cE_r=\cF
 \]
 where each quotient $\cQ_i:=\cE_{i}/\cE_{i-1}$ is a pure sheaf with irreducible support, whose restriction to its support is (torsion-free) of rank $1$. This result was proved in \cite[Theorem 1.2]{Tipler25} (under the extra assumption that $X$ is smooth) with a different approach using Perling's families of multifiltrations. The injections $\cE_i\to\cE_{i+1}$ are called \textit{elementary injections} in \cite{Tipler25}. Be careful though that the quotients $\cQ_i$ correspond to elementary toric sheaves in the terminology of the present paper only when $(\cQ_i)_{\vert \Supp(\cQ_i)}$ is reflexive, that is when the injection $\cE_i\to\cE_{i+1}$ is \textit{saturated} (cf \cite[Definition 4.9]{Tipler25}). The proof presented in this paper provides a more conceptual and geometric explanation to the factorisation result \cite[Theorem 1.2]{Tipler25}.
\end{remark}
Corollary \ref{corintro:decomposition theorem} follows directly from the proof of Theorem \ref{theointro:decomposition theorem}.
\begin{proof}[Proof of Corollary \ref{corintro:decomposition theorem smooth case}]
By Corollary \ref{corintro:decomposition theorem}, it is enough to show that the property $P$ is satisfied by all the elementary toric sheaves on $X$. Let $\cE\in\Coh^T(X)$ be elementary. By Lemma \ref{lem:elementary toric sheaf geometric}, we may assume that
\[
\cE\simeq (\iota_*\cO_{V}\otimes_{\cO_X}\cL)[m],
\]
where $\iota:V\to X$ is the inclusion of some orbit closure $V=V(\tau)$ for $\tau\in\Sigma$, $\cL$ is a toric line bundle on $X$, and $m\in M$. We will assume $m=0$, which will not affect the argument. The key fact is that $\iota_*\cO_{V}$ admits a Koszul resolution by direct sums of toric line bundles. Indeed, the orbit-cone correspondence gives
\[
V(\tau)=\bigcap_{\rho\in\tau(1)}\, D_\rho.
\]
Then, we can introduce the toric vector bundle of rank $r=\codim V$ :
\[
 \cF=\bigoplus_{\rho\in\tau(1)}\, \cO_X(D_\rho)
\]
and find a section $s=(s_\rho)_{\rho\in \tau(1)}\in H^0(X,\cF)$, with $(s_\rho)$ defining sections for the $D_\rho$'s, such that $V(\tau)=\lbrace s= 0 \rbrace$. Convolution by $s$ induces the Koszul complex
\[
 0\to \Lambda^r\cF^\vee \to \ldots \to \Lambda^2\cF^\vee \to\cF^\vee\to \cO_X\to \iota_*\cO_V\to 0
\]
which corresponds to the toric locally free resolution of $\iota_*\cO_V$ :
\[
 0\to \cO_X(-\sum_{\rho\in\tau(1)}\, D_\rho) \to \ldots \to \bigoplus_{\sigma\in\tau(2)}\, \cO_X(-\sum_{\rho\in\sigma(1)}\, D_\rho)\to\bigoplus_{\rho\in\tau(1)}\, \cO_X(-D_\rho) \to \cO_X\to \iota_*\cO_V\to 0.
\]
Tensoring by $\cL$ gives a toric resolution of $\cE$ by direct sums of toric line bundles. As the property $P$ is satisfied by toric line bundles, and preserved in short exact sequences, a simple induction shows that it is satisfied by each term in this resolution of $\cE$, and then by $\cE$ itself. The result follows.
\end{proof}

\section{Application : cohomology}
\label{sec:cohomology}

We consider in this section a normal and complete toric variety $X$ with fan $\Sigma$. We also fix $\E:=(E^\Sigma,\eta^\Sigma)$ (also denoted sometimes $(E^\sigma_m)$) to be a $\Sigma$-family associated to  a quasicoherent $T$-sheaf $\cE$ over $X$.
\subsection{Cohomology of $\Sigma$-families}
By considering the \v{C}ech complex associated to the affine cover of $X$ by the $U_\sigma$'s, for $\sigma\in\Sigma(n)$ running through maximal dimensional cones, one obtains that the cohomology groups $H^i(X,\cE)$ are $M$-graded. We introduce now a complex whose cohomology will  compute each graded piece.
As in \cite{Brion97Tohoku,BressLunts}, we fix an orientation on each cone $\sigma\in\Sigma$ and signs $\varepsilon_{\tau\sigma}$ for each pair $\tau\leq\sigma$, $\dim\sigma=\dim\tau+1$, with $\varepsilon_{\tau\sigma}= +1$ when the orientation of $\sigma$ induces the orientation of $\tau$ by restriction, and $\varepsilon_{\tau\sigma}= -1$ otherwise. We then define $(C^\bullet(\Sigma, \E)_m, \del_m^\bullet)$ for each $m\in M$ as follows. For each $0\leq i\leq n$, set the cochains to be
\[
C^i(\Sigma,\E)_m:=\bigoplus_{\sigma\in\Sigma(n-i)}\, E^\sigma_m.
\]
The differentials are given, for $0\leq i\leq n-1$, by
\begin{equation}
 \label{eq:diff Klyachko Cech complex}
\begin{array}{cccc}
 \del^i_m : & C^i(\Sigma,\E)_m &\to & C^{i+1}(\Sigma,\E)_m\\
            & (s_\sigma)_{\sigma\in\Sigma(n-i)} & \mapsto & \displaystyle\sum_{\tau\in\Sigma(n-i-1)} \sum_{\substack{\tau\leq\sigma\\\sigma\in \Sigma(n-i)}}\, \varepsilon_{\tau\sigma} \, \eta_{\tau\sigma}\circ\iota^*_{\tau\sigma}s_\sigma
\end{array}
\end{equation}
and by 
\[\del^n_m : C^n(\Sigma,\E)_m\to 0,\]
and 
\[\del^{-1}_m: 0 \to C^0(\Sigma,\E)_m,\]
where recall that $\eta_{\tau\sigma}\circ\iota^*_{\tau\sigma}: E^\sigma_m\to E^\tau_m$ is induced by the restriction map, cf Definition \ref{def:Sigma families}.
We leave to the reader the verification that $\del^{i+1}_m\circ\del^i_m=0$. Following \cite{AlHoWi}, we will call the complex $(C^\bullet(\Sigma,\E)_m,\del^\bullet_m)$ the $m$-th Klyachko--\v{C}ech complex. We can introduce the associated cohomology groups
\[
H^i(\Sigma,\E)_m:=\frac{\ker \del^i_m}{\Ima\, \del^{i-1}_m}.
\]
Our goal is now to prove the following result, due to Klyachko in the locally free case, over a smooth base \cite{Kly90}.
\begin{theorem}
 \label{theo:computing cohomology of toric sheaves}
 Assume that $\cE$ is a coherent toric sheaf over a normal and complete toric variety $X$. Let $\E=(E^\Sigma,\eta^\Sigma)$ be the associated $\Sigma$-family. Then,
 \[
 H^i(X,\cE)\simeq \bigoplus_{m\in M} H^i(\Sigma,\E)_m.
 \]
\end{theorem}
We first present some corollaries, following \cite{Kly90}.  For each $m\in M$, set the weight-$m$ Euler characteristic :
\[
\chi(\cE)_m:=\sum_{i=0}^n\, (-1)^i \dim\big(H^i(X,\cE)_m\big).
\]
\begin{corollary}
 \label{cor:weighted euler characteristic from cohomology}
 Assume that $\cE$ is a coherent toric sheaf over a normal and complete toric variety $X$, with associated $\Sigma$-family $\E=(E^\sigma_m)$. Then for each $m\in M$,
 \[\chi(\cE)_m= \sum_{\sigma\in\Sigma}\, (-1)^{\codim(\sigma)}\, \dim(E^\sigma_m).\]
\end{corollary}
When $\cE$ is torsion-free, each vector space $E^\sigma_m$ can be seen as a subspace of the inductive limit \[E=\displaystyle\lim_{\to} E^\sigma_m.\]
\begin{corollary}
 \label{cor:H0 and Hn from complex family multifiltrations}
 Let $\cE$ be a torsion-free toric sheaf on a complete and normal toric variety $X$, with associated $\Sigma$-family $\E=(E^\sigma_m)$. Then for each $m\in M$,
 \[H^0(X,\cE)_m=\bigcap_{\sigma\in\Sigma(n)} \, E^\sigma_m\]
 and 
 \[
H^n(X,\cE)_m=E/\sum_{\rho\in\Sigma(1)}E^\rho_m.
\]
\end{corollary}

\subsection{Proof of Theorem \ref{theo:computing cohomology of toric sheaves}} We will use a sheaf version of  $(C^\bullet(\Sigma,\E)_m,\del_m)$. We define, for each $0\leq i\leq n$ :
\[
\cC^i(\Sigma,\cE):=\bigoplus_{\sigma\in\Sigma(n-i)}\, (\iota_\sigma)_*(\cE_{\vert U_\sigma}),
\]
where $\iota_\sigma:U_\sigma\to X$ is the inclusion map. This is naturally a quasicoherent toric sheaf. More precisely, for each $\tau\in\Sigma$, and $m\in M$, we have for each $\sigma\in\Sigma$ :
\[
\begin{array}{ccc}
 \Gamma\big(U_\tau,(\iota_\sigma)_*(\cE_{\vert U_\sigma})\big)_m & = & \Gamma(U_\tau\cap U_\sigma,\cE_{\vert U_\sigma})_m \\
               &  = & \Gamma(U_\tau\cap U_\sigma,\,\cE)_m \\
               & =& E^{\tau\cap\sigma}_m,
\end{array}
\]
and then 
\begin{equation}
 \label{eq:Gamma Usigma acyclic resolution}
\Gamma(U_\tau,\cC^i(\Sigma,\cE))_m=\bigoplus_{\sigma\in\Sigma(n-i)}\, E^{\tau\cap\sigma}_m.
\end{equation}
We can define as in Equation \eqref{eq:diff Klyachko Cech complex} differentials $\del^i :  \cC^i(\Sigma,\cE) \to \cC^{i+1}(\Sigma,\cE)$, and
an injective morphism $\iota :\cE\to \cC^0(\Sigma,\cE)$ given by the direct sum over $\Sigma(n)$ of the restriction maps $\iota_\sigma$. We leave again to the reader checking that the following is a complex :
\begin{equation}
 \label{eq:acyclic resolution coherent toric sheaves}
0\to \cE \xrightarrow{\iota} \cC^0(\Sigma,\cE)\xrightarrow{\del^0} \cC^1(\Sigma,\cE)\xrightarrow{\del^1} \ldots  \xrightarrow{\del^{n-1}} \cC^n(\Sigma,\cE)\to 0.
\end{equation}
 To prove Theorem \ref{theo:computing cohomology of toric sheaves},  we will show that \eqref{eq:acyclic resolution coherent toric sheaves} is an acyclic resolution for $\cE$, assuming finiteness of the $\Sigma$-family $\E$. First, as $\cE$ is a sheaf, $\Ima\,\iota= \ker\, \del^0$.  We then have acyclicity :
\begin{lemma}
 \label{lem:acyclicity resolution of E}
  Let $\cE$ be a quasicoherent toric sheaf on $X$. For each $0\leq i\leq n$ and each $j\geq 1$, one has
 \[
 H^j(X,\cC^i(\Sigma,\cE))=0.
 \]
\end{lemma}
\begin{proof}
 It is  enough to show that for each $\sigma$, $(\iota_\sigma)_*(\cE_{\vert U_\sigma})$ has vanishing cohomology in degree greater or equal to one. By the \v{C}ech description of cohomology, one deduce that
 $H^j(X,(\iota_\sigma)_*(\cE_{\vert U_\sigma})\simeq H^j(U_\sigma,\iota_\sigma^*\cE)$, and the result follows from Serre vanishing Theorem \cite[Theorem III.3.7]{Hartshorne}.
\end{proof}
To prove exactness of the complex \eqref{eq:acyclic resolution coherent toric sheaves}, we will use Corollary \ref{corintro:decomposition theorem}.
In what follows, we let $(\cC^\bullet(\Sigma,\cE),\del^\bullet)$ to stand for the complex \eqref{eq:acyclic resolution coherent toric sheaves}.
\begin{lemma}
 \label{lem:exact seq toric sheaves gives exact seq acyclic resolution}
 Any exact sequence
 \[
 0\to\cE \to\cF\to\cG\to 0
 \]
  of quasicoherent toric sheaves induces an exact sequence of complexes :
  \begin{equation}
   \label{eq:exact sequence complexes of sheaves}
  0\to (\cC^\bullet(\Sigma,\cE),\del^\bullet) \to (\cC^\bullet(\Sigma,\cF),\del^\bullet) \to (\cC^\bullet(\Sigma,\cG),\del^\bullet) \to 0.
  \end{equation}
  Moreover, if two of the complexes $(\cC^\bullet(\Sigma,\cE),\del^\bullet)$, $(\cC^\bullet(\Sigma,\cF),\del^\bullet)$ or $(\cC^\bullet(\Sigma,\cG),\del^\bullet)$ are exact, so is the third one.
\end{lemma}
\begin{proof}
This is again straightforward. Denote by $(E^\sigma_m)$, $(F^\sigma_m)$ and $(G^\sigma_m)$ the $\Sigma$-families of $\cE$, $\cF$ and $\cG$, respectively.
For each $\sigma\in\Sigma$, $\tau\in\Sigma$ and $m\in M$, restriction to $U_{\tau\cap\sigma}$ yields exactness of 
\[
0\to E^{\tau\cap\sigma}_m \to F^{\tau\cap\sigma}_m \to G^{\tau\cap\sigma}_m\to 0.
\]
Then, by the description of the sections of $\cC^i(\Sigma,\cE)$, cf \eqref{eq:Gamma Usigma acyclic resolution}, we have induced exact sequences
\[
0\to \Gamma(U_\tau,\cC^i(\Sigma,\cE))_m\to \Gamma(U_\tau,\cC^i(\Sigma,\cF))_m\to \Gamma(U_\tau,\cC^i(\Sigma,\cG))_m\to 0.
\]
Commutativity of those sequences with the differentials $\del^\bullet$ follows from the fact that morphisms of sheaves commute with restriction maps, by definition. The  last statement follows by considering the long exact sequence associated to \eqref{eq:exact sequence complexes of sheaves}.
\end{proof}
The following proposition is then the key to prove Theorem \ref{theo:computing cohomology of toric sheaves}.
\begin{proposition}
 \label{prop:acyclic resolution for E}
 Assume that $X$ is normal and complete. Let $\cE$ be a coherent toric sheaf on $X$. Then the complex $(\cC^\bullet(\Sigma,\cE),\del^\bullet)$ is exact.
\end{proposition}
\begin{proof}
 Using Corollary \ref{corintro:decomposition theorem}, and Lemma \ref{lem:exact seq toric sheaves gives exact seq acyclic resolution}, it is enough to show the result when $\cE$ is an elementary toric sheaf. Then, there is a torus invariant subvariety $\iota:V\to X$, with fan $\Sigma_V$,  and a rank one reflexive sheaf $\cL$ on $V$, such that $\cE=\iota_*\cL$. We leave to the reader the proof of the following facts :
 \begin{itemize}
  \item For $i>\dim(V)$, $\cC^i(\Sigma,\iota_*\cL)=0$,
  \item For $i\leq \dim(V)$, $\iota^*\cC^i(\Sigma,\iota_*\cL)=\cC^i(\Sigma_V,\cL)$, and $\del^i\circ\iota^*=\iota^*\circ\del^i$,
  \item For each $i$, $\cC^i(\Sigma,\iota_*\cL)$ has support in $V$.
 \end{itemize}
 From those properties, one deduces that if the complex $(\cC^\bullet(\Sigma_V,\cL),\del^\bullet)$ is exact then  so is $(\cC^\bullet(\Sigma,\iota_*\cL),\del^\bullet)$. Hence, we can reduce ourselves to the case when $\cE=\cL$ is a rank one reflexive toric sheaf on $X$, and there exists integers $(a_\rho)_{\rho\in\Sigma(1)}\in\Z^{\vert\Sigma(1)\vert}$ such that
 \[
 \cL\simeq \cO_X\big(\sum_{\rho\in\Sigma(1)}\,a_\rho\,D_\rho\big).
 \]
 We set $(L^\sigma_m)$ to be a family of multifiltrations of $\cL$. By proposition \ref{prop:family of filtrations rank one reflexive sheaves}, for any $\sigma\in\Sigma$, and any weight $m\in M$, we may assume that
 \[
 L^\sigma_m=\left\{
 \begin{array}{ccc}
  \lbrace 0\rbrace & \mathrm{ if } & \exists\, \rho\in\sigma(1)\,,\: \langle m, u_\rho\rangle < -a_\rho\\
  \K & \mathrm{ if } & \forall\, \rho\in\sigma(1)\,,\: \langle m, u_\rho\rangle \geq -a_\rho.
 \end{array}
 \right.
 \]
 We wish to prove that $(\cC^\bullet(\Sigma,\cL),\del^\bullet)$ is an exact sequence of quasicoherent toric sheaves. By Theorem \ref{theorem:equiv categories Tsheaves and Sigma families}, it is enough to prove exactness at the level of $\Sigma$-families. This amounts to showing exactness, for each $\tau\in\Sigma$, and $m\in M$, of the complex $(\Gamma(U_\tau,\cC^\bullet(\Sigma,\cL))_m,\del^\bullet)$.
 By \eqref{eq:Gamma Usigma acyclic resolution}, this complex is
 \begin{equation}
 \label{eq:cohomology complex acyclic resolution rank one case}
 0\to L^\tau_m\to \bigoplus_{\sigma\in\Sigma(n)}\, L^{\tau\cap\sigma}_m \xrightarrow{\del_m^0}  \ldots  \xrightarrow{\del_m^{n-2}}\bigoplus_{\rho\in\Sigma(1)} L^{\tau\cap\rho}_m \xrightarrow{\del_m^{n-1}} \K \to 0
\end{equation}
with 
\[
(\del^i_m (s_\sigma))_{\sigma'}= \sum_{\substack{\sigma'\leq\sigma\\\sigma\in \Sigma(n-i)}}\, \varepsilon_{\sigma'\sigma} \, s_\sigma.
\]
 From the description of $(L^\sigma_m)$, the contribution of $L^{\tau\cap\sigma}_m$ to this latter complex will be non-zero precisely when for all $\rho\in\tau(1)\cap\sigma(1),\, \langle m, u_\rho\rangle \geq -a_\rho$. Let then $\Sigma_{\hat\tau_m}\subset\Sigma$ be the subset of cones $\sigma\in\Sigma$ such that for all $\rho\in \sigma(1)$, either $\rho\notin \tau(1)$, or $\rho\in\tau(1)$ and $\langle m, u_\rho \rangle \geq -a_\rho$. Equivalently, $\Sigma_{\hat\tau_m}$ is the largest subfan of $\Sigma$ that do not contain the set of rays  $\tau_m:=\lbrace \rho\in\tau(1)\:\vert\: \langle m, u_\rho\rangle < -a_\rho \rbrace\subset \tau(1)$. Then, we reduced ourselves to showing exactness of the complex
 \begin{equation}
  \label{eq:augmented homology proof of acyclic resolution}
 0\to L_m^{\tau_m}\to \bigoplus_{\sigma\in\Sigma_{\hat\tau_m}(n)}\, \K_\sigma \xrightarrow{\del_m^0}  \ldots \to\bigoplus_{\rho\in\Sigma_{\hat\tau_m}(1)} \K_\rho \xrightarrow{\del_m^{n-1}} \K \to 0,
 \end{equation}
 where each $\K_\sigma$ is a copy of $\K$. We recognize the augmented homology complex for the CW-complex given by $\Sigma_{\hat\tau_m}\cap S^{n-1}$, where $S^{n-1}\subset N_\R$ is the sphere given by a fixed euclidean structure on $N_\R$. We further note that either $\tau_m=\varnothing$, and then $\vert\Sigma_{\hat\tau_m}\vert\cap S^{n-1}=S^{n-1}$, or $\tau_m\neq 0$, and then  $\vert\Sigma_{\hat\tau_m}\vert\cap S^{n-1}$ is contractible. Indeed, for each cone $\sigma\in\Sigma$ such that $\sigma(1)\cap\tau_m\neq\varnothing$, the cell $\sigma\cap S^{n-1}$ admits a deformation retract to a point in $\tau_m\cap S^{n-1}$, and $\tau\cap S^{n-1}$ is contractible.  In both cases, this proves exactness of \eqref{eq:augmented homology proof of acyclic resolution}, which concludes the proof.
\end{proof}

We can now conclude the proof of Theorem \ref{theo:computing cohomology of toric sheaves}. Indeed, by Proposition \ref{prop:acyclic resolution for E}, $(\cC^\bullet(\Sigma,\cE),\del^\bullet)$ is exact. By Lemma \ref{lem:acyclicity resolution of E}, for each $0\leq i\leq n$, $\cC^i(\Sigma,\cE)$ have vanishing cohomology in positive degrees. Then, it is a standard fact that the global sections of an acyclic resolution computes cohomology for a given sheaf. The result then follows by construction of $\cC^\bullet(\Sigma,\cE)$.

\begin{remark}
\label{rem:quasicoherent extension of cohomology for toric sheaves}
One can extend further Theorem \ref{theo:computing cohomology of toric sheaves} to the quasicoherent case, by following the results in \cite[Section 3]{Huang}.  Another proof can also be obtained in the simplicial case by using a little more algebra, the \textit{Canonical \v{C}ech complex} of \cite[Section 6, Example 6.6]{Miller}, together with Batyrev and Cox presentation of toric sheaves \cite{BaCo}. We refer the
interested reader to \cite[Proposition 1.3.14]{Molto}. Finally, we point out that the proof of Theorem \ref{theo:computing cohomology of toric sheaves} does not reduce to a simple \v{C}ech cohomology computation, as claimed in \cite{KnuSha98}, unless $X$ is of Picard rank one.
\end{remark}

 \subsection{Alexander and Serre dualities}
\label{sec:alexander and serre dualities}
We finish this section by explaining how Alexander duality implies Serre duality for toric line bundles, assuming $X$ to be smooth. Let
\[
\cL=\cO_X\big(\sum_{\rho\in\Sigma(1)}\,a_\rho\,D_\rho\big)
\]
be a toric line bundle, for some integers $(a_\rho)_{\rho\in\Sigma(1)}\in\Z^{\vert\Sigma(1)\vert}$, with family of multifiltrations
$\LL=(L^\sigma_m)$.
Set $\check\cL=\cL^\vee\otimes \cK_X$, and set $\check \LL=(\check L^\sigma_m)$ be the associated family of multifiltrations. Using \cite[Section 2.1]{DDK20}, one obtains
\[
 \check L^\sigma_{-m}=\left\{
 \begin{array}{ccc}
  \lbrace 0\rbrace & \mathrm{ if } & \exists\, \rho\in\sigma(1)\,,\: \langle m, u_\rho\rangle \geq -a_\rho\\
  \K & \mathrm{ if } & \forall\, \rho\in\sigma(1)\,,\: \langle m, u_\rho\rangle < -a_\rho.
 \end{array}
 \right.
\]
We then introduce the subfans of $\Sigma$ :
\[
\Sigma_{\cL,m}=\lbrace \sigma\in\Sigma\:\vert\: \forall\, \rho\in\sigma(1),\, \langle m, u_\rho\rangle \geq -a_\rho \rbrace
\]
and 
\[
\Sigma_{\check\cL,-m}=\lbrace \sigma\in\Sigma\:\vert\: \forall\, \rho\in\sigma(1),\, \langle m, u_\rho\rangle < -a_\rho \rbrace.
\]
From the proof of Proposition \ref{prop:acyclic resolution for E}, the complex $(C^\bullet(\Sigma,\check \LL)_{-m},\del^\bullet_{-m})$ computes the reduced homology of $\check K:=S^{n-1}\cap \vert \Sigma_{\check\cL,-m}\vert$. On the other hand, it is straightforward to check that the cohomology of the dual complex of $(C^\bullet(\Sigma,\LL)_m,\del^\bullet_m)$ computes the reduced cohomology of  $K:=S^{n-1}\cap \vert \Sigma_{\cL,m}\vert$. As the rays of $\Sigma$ all belong to $\Sigma_{\check\cL,-m}(1)\cup \Sigma_{\cL,m}(1)$, the space $K$ is a deformation retract of $S^{n-1}\setminus \check K$. Hence, the reduced (co)homology of $K$ is isomorphic to the one of $S^{n-1}\setminus \check K$. By Alexander duality, we deduce that the $i$-th cohomology for the dual complex of $(C^\bullet(\Sigma,\LL)_m,\del^\bullet_m)$ is isomorphic to the $(n-i)$-th cohomology for $(C^\bullet(\Sigma,\check \LL)_{-m},\del^\bullet_{-m})$, recovering Serre duality in degrees $1\leq i\leq n-1$. In general, Serre duality for locally free toric sheaves may be viewed as a version of Alexander duality with local coefficient systems.

\bibliography{Toric}

@misc{AlHoWi,
 author = {Altmann, Klaus and Hochenegger, Andreas and Witt, Frederik},
 title = {Toric sheaves and polyhedra},
 year = {2024},
 howpublished = {Preprint, {arXiv}:2412.03476 [math.{AG}] (2024)},
 url = {https://arxiv.org/abs/2412.03476},
 arXiv = {arXiv:2412.03476}
}

@article{BaCo,
 author = {Batyrev, Victor V. and Cox, David A.},
 title = {On the {Hodge} structure of projective hypersurfaces in toric varieties},
 fjournal = {Duke Mathematical Journal},
 journal = {Duke Math. J.},
 issn = {0012-7094},
 volume = {75},
 number = {2},
 pages = {293--338},
 year = {1994},
 language = {English},
 doi = {10.1215/S0012-7094-94-07509-1},
 zbMATH = {706690},
 Zbl = {0851.14021}
}

@article{BressLunts,
 author = {Bressler, Paul and Lunts, Valery A.},
 title = {Intersection cohomology on nonrational polytopes},
 fjournal = {Compositio Mathematica},
 journal = {Compos. Math.},
 issn = {0010-437X},
 volume = {135},
 number = {3},
 pages = {245--278},
 year = {2003},
 language = {English},
 doi = {10.1023/A:1022232232018},
 zbMATH = {1901512},
 Zbl = {1024.52005}
}

@article{Brion97Tohoku,
 author = {Brion, Michel},
 title = {The structure of the polytope algebra},
 fjournal = {T{\^o}hoku Mathematical Journal. Second Series},
 journal = {T{\^o}hoku Math. J. (2)},
 issn = {0040-8735},
 volume = {49},
 number = {1},
 pages = {1--32},
 year = {1997},
 language = {English},
 zbMATH = {1036420},
 Zbl = {0881.52008}
}

@book{CLS,
  title = {Toric Varieties},
  author = {Cox, D. and Little, J. and Schenck, H.},
  isbn = {9780821884263},
  series = {Graduate studies in mathematics},
  year = {2011},
  publisher = {American Mathematical Soc.}
}

@article{DDK20,
  title = {Stabiliy of equivariant vector bundle over toric varieties},
  author = {Dasgupta, J. and Dey, A. and Khan, B.},
  journal = {Documenta Mathematica},
  volume = {25},
  pages = {~1787-1833},
  year = {2020}
}

@book{Hartshorne,
Author = {Hartshorne, R.},
Title = {Algebraic geometry},
Series = {Graduate Texts in Mathematics},
number = {52.},
publisher = {New York-Heidelberg-Berlin: Springer-Verlag.},
pages = {496 p.},
year ={(1977).}
}

@article{Har80,
  title = {Stable Reflexive Sheaves},
  author = {Hartshorne, R.},
  year = {1980},
  pages = {~121 - 176},
  volume = {254},
  journal = {Mathematische Annalen}
}

@book{HuLe, 
 place={Cambridge}, 
 edition={2}, 
 series={Cambridge Mathematical Library},
  title={The Geometry of Moduli Spaces of Sheaves}, 
  DOI={10.1017/CBO9780511711985}, 
  publisher={Cambridge University Press}, 
  author={Huybrechts, D. and Lehn, M.}, year={2010}, 
  collection={Cambridge Mathematical Library}
  }

@inbook{Huang,
author = {Huang, Mu-Wan},
title = {Fan Cohomology and Equivariant Chow Rings of Toric Varieties},
year = {2009},
publisher = {University of Nebraska - Lincoln},
series = {Department of
Mathematics: Dissertations, Theses, and Student Research}
}

@article{Kly90,
  title = {Equivariant bundle on toral varieties},
  author = {Klyachko, A.},
  year = {1990},
  pages = {~337 - 375},
  volume = {35},
  number = {2},
  journal = {Mathematics of the USSR-Izvestiya}
}

@Article{KnuSha98,
 Author = {Knutson, A. and Sharpe, E.},
 Title = {Sheaves on toric varieties for physics},
 FJournal = {Advances in Theoretical and Mathematical Physics},
 Journal = {Adv. Theor. Math. Phys.},
 ISSN = {1095-0761},
 Volume = {2},
 Number = {4},
 Pages = {873--961},
 Year = {1998},
 Language = {English},
 DOI = {10.4310/ATMP.1998.v2.n4.a6},
 zbMATH = {1324687},
 Zbl = {0943.14021}
}

@article{Miller,
 author = {Miller, Ezra},
 title = {The {Alexander} duality functors and local duality with monomial support},
 fjournal = {Journal of Algebra},
 journal = {J. Algebra},
 issn = {0021-8693},
 volume = {231},
 number = {1},
 pages = {180--234},
 year = {2000},
 language = {English},
 doi = {10.1006/jabr.2000.8359},
 zbMATH = {1525271},
 Zbl = {0968.13009}
}

@article{Molto, 
title = {Vector bundles and sheaves on toric varieties}, 
author = {Salat Moltó, Martí}, 
year = {2022}, 
publisher = {Universitat de Barcelona}, 
url = {http://hdl.handle.net/10803/688999} }

@Article{NaTip,
 Author = {Napame, A. and Tipler, C.},
 Title = {Toric sheaves, stability and fibrations},
 FJournal = {Journal of Pure and Applied Algebra},
 Journal = {J. Pure Appl. Algebra},
 ISSN = {0022-4049},
 Volume = {228},
 Number = {3},
 Pages = {30},
 Note = {Id/No 107494},
 Year = {2024},
 Language = {English},
 DOI = {10.1016/j.jpaa.2023.107494},
 zbMATH = {7755438},
 Zbl = {1527.14103}
}

@article{Payne09,
 author = {Payne, Sam},
 title = {Toric vector bundles, branched covers of fans, and the resolution property},
 fjournal = {Journal of Algebraic Geometry},
 journal = {J. Algebr. Geom.},
 issn = {1056-3911},
 volume = {18},
 number = {1},
 pages = {1--36},
 year = {2009},
 language = {English},
 doi = {10.1090/S1056-3911-08-00485-2},
 zbMATH = {5480740},
 Zbl = {1161.14039}
}

@article{Per04,
  title = {Graded Rings and Equivariant Sheaves on Toric Varieties},
  author = {Perling, M.},
  year = {2004},
  pages = {~181–197},
  volume = {263/264},
  journal = {Mathematische Nachrichten}
}

@misc{Tipler25,
 author = {Carl Tipler},
 title = {Stable toric sheaves. {I} : {Chern} classes},
 year = {2025},
 howpublished = {Preprint, {arXiv}:2510.14651 [math.{AG}] (2025)},
 url = {https://arxiv.org/abs/2510.14651},
 arXiv = {arXiv:2510.14651}
}
\bibliographystyle{amsplain}

\end{document}